\documentclass[preprint]{elsarticle}

\usepackage{amssymb}
\usepackage{amsmath}
\usepackage{amsthm}
\usepackage{amsfonts}
\usepackage{paralist}
\usepackage{natbib}
\usepackage{indentfirst}
\usepackage{graphics} 
\usepackage{epsfig} 
\usepackage{colortbl}
\biboptions{numbers,sort&compress}

\newtheorem{theorem}{Theorem}[section]

\newtheorem{lemma}[theorem]{Lemma}

\newtheorem{assumption}[theorem]{Assumption}

\newtheorem{example}{Example}
\newtheorem{remark}{Remark}

\DeclareMathOperator{\trace}{trace}

\DeclareMathOperator{\D}{\Delta}
\journal{Journal of \LaTeX\ Templates}

\begin{document}

\begin{frontmatter}

\title{The rate of $L^p$-convergence for the Euler-Maruyama method of the stochastic differential equations with Markovian switching \tnoteref{mytitlenote}}
\tnotetext[mytitlenote]{This research was supported by NSFC (Grant No.12071101 and No.11671113).}

\author{Minghui Song\corref{cor1}}
\ead{songmh@hit.edu.cn}
\author{Yuhang Zhang}
\ead{19b912028@stu.hit.edu.cn}
\author{Mingzhu Liu}
\ead{mzliu@hit.edu.cn}
\address{School of Mathematics, Harbin Institute of Technology, Harbin, 150001, China}
\cortext[cor1]{Corresponding author}

\begin{abstract}

This work deals with the Euler-Maruyama (EM) scheme for stochastic differential equations with Markovian switching (SDEwMSs). We focus on the $L^p$-convergence rate  $(p\ge 2)$ of the EM method given in this paper. As far as we know, the skeleton process of the Markov chain is used in the continuous numerical methods in most papers. By contrast, the continuous EM method in this paper is to use the Markov chain directly. To the best of our knowledge, there are only two papers that consider the rate of $L^p$-convergence, which is no more than $1/p~(p \ge 2)$ in these papers. The contribution of this paper is that the rate of $L^p$-convergence of the EM method can reach $1/2$. We believe that the technique used in this paper to construct the EM method can also be used to construct other methods for SDEwMSs. 
\end{abstract}

\begin{keyword}
stochastic differential equations \sep Markov chain\sep Euler-Maruyama method \sep $L^p$-convergence \sep convergence rate
\MSC[2010] 65C30\sep  60H35
\end{keyword}

\end{frontmatter}

\section{Introduction}\label{sec1}

Stochastic differential equations with Markovian switching (SDEwMSs), also known as hybrid stochastic differential equations, play an important role in stochastic theory and have been used in various fields, such as the theory of control and neural networks (\cite{2006XMao,1994Birkhauser,1990Mariton}).  
Most of SDEwMSs do not have explicit solutions so it is important to have numerical solutions (\cite{YUAN2004223,NGUYEN20121170,doi:10.1137/16M1084730,NGUYEN201814,KUMAR2020112917,GAO2021111224,MAO2007936,BAO20091379,doi:10.1137/080727191,ShaoboZhou2015Calcolo,OBRADOVIC2019664,DENG201915}).  \cite{YUAN2004223} is the first research that developed the Euler-Maruyama (EM) scheme for SDEwMSs with the global Lipschitz continuous coefficients and considered the $L^2$-convergence rate for EM solutions.  \cite{doi:10.1137/16M1084730} designed approximation methods of Milstein type for SDEwMSs and proved the convergence rate is better than the generally adopted EM procedures.

The primary motivation for this work came from the following observation: to our knowledge, there are 
plenty of papers on the convergence of numerical algorithms for hybrid systems, most of them showed the convergence (without order)(e.g., \cite{MAO2007936,BAO20091379,doi:10.1137/080727191,ShaoboZhou2015Calcolo,OBRADOVIC2019664,DENG201915,doi:10.1080/10236190802695456}) or the rate of convergence in the sense of pathwise or mean square (e.g., \cite{YUAN2004223,NGUYEN20121170,doi:10.1137/16M1084730,NGUYEN201814,KUMAR2020112917,GAO2021111224,HOANG2014822,Congyuhao2019IJCM}). However, there are only a few papers revealed the $L^p$-convergence order of  numerical methods for hybrid systems (\cite{2019HouZhenting, 2020ZhangWei}). Not only that, the order of $L^p$-convergence for Euler type numerical algorithms proved in these papers are no more than $1/p~(p\ge 2)$, instead of the well known $1/2$. To be specific, the main result in \cite{2019HouZhenting}  (Theorem 3) shows
  \begin{equation}\label{e: remark 1}
  \mathbb{E}\sup_{t\in[0,T]}|y(t)-x(t)|^p\le C_5\D(1+\mathbb{E}|x_0|^p),
  \end{equation}
where $x(t)$ is the exact solution of the stochastic delay differential equation with phase semi-Markovian switching and Poisson jumps, $y(t)$ denotes the numerical approximation using the continuous $\theta$ method, $\D$ is the given step-size, $C_5$ denotes a generic constant that independent of  $\D$, $x_0$ is the initial data. By analyzing the details in this paper, we find that the problem first appears in the estimations of 
  \begin{equation*}
    \mathbb{E}\int_0^T|f(Z_1(s),Z_1(s-\tau),{\color{red} r_1(s)})-f(Z_1(s),Z_1(s-\tau),{\color{red} r(s)})|^p{\rm d}s,
 \end{equation*}
and 
  \begin{equation*}
    \mathbb{E}\int_0^T|f(Z_2(s),Z_2(s-\tau),{\color{red} r_2(s)})-f(Z_2(s),Z_2(s-\tau),{\color{red} r(s)})|^p{\rm d}s
 \end{equation*}
(Lemma 4 in \cite{2019HouZhenting}), we think these two terms can be seen as the errors in approximating $r(s)$ by $r_1(s)$ and $r_2(s)$, where $r(s)$ is the given continuous-time Markov chain. Similar estimations also exist in many works aforementioned, for example, Eq.(3.7) in \cite{YUAN2004223}, Lemma 3 in \cite{MAO2007936}, Eq.(3.6) in \cite{BAO20091379}, as well as Corollary 3.1 in \cite{2020ZhangWei}, etc. Based on this fact, the main idea of this work is to use $r(s)$ itself to construct a numerical scheme, rather than its approximation. Therefore, the innovations of this paper are as follows:
\begin{itemize}
\item We will use the continuous-time Markov chain itself to develop the numerical scheme, instead of its approximation.
\item The order of $L^p$-convergence for the EM method given in this work to SDEwMSs can reach $1/2$.
\end{itemize}

The rest of the paper is arranged as follows. In Section \ref{Notations}, we present some notations and fundamental assumptions, moreover, we further introduce the classical EM method for SDEwMSs which is often used in literatures. Then we develop a different EM scheme in Section \ref{Euler method}. The rate of $L^p$-convergence for the EM method be proved in Section \ref{Rate of strong convergence}.  Finally, we give the conclusion of this paper in Section \ref{Conclusions}.

\section{Notations and preliminaries}\label{Notations}

In the rest of this work, except as otherwise noted, we let $(\Omega ,\mathcal{F},\left\{\mathcal{F}_t\right\}_{t\ge 0},\mathbb{P})$ be a complete probability space with a filtration $\left\{\mathcal{F}_t\right\}_{t\ge 0}$ which satisfies the general conditions (namely, it is right continuous and $\mathcal{F}_0$ involves all $\mathbb{P}$-null sets). The transpose of $A$ is denoted by $A^{\rm T}$ when $A$ is a vector or matrix. $B(t)=(B_1(t),\dots,B_d(t))^{\rm T}$ represents a $d$-dimensional Brownian motion defined on the $(\Omega ,\mathcal{F},\left\{\mathcal{F}_t\right\}_{t\ge 0},\mathbb{P})$. If $x$ is a vector, $\vert  x\vert$ denotes its Euclidean norm. $\vert A\vert =\sqrt{\trace(A^{\rm T}A)}$ denotes the trace norm of a matrix $A$. If $u$ and $v$ are two real numbers, let $u\vee v$ and $u\wedge v$ be $\max\left\{u,v\right\}$ and $\min\left\{u,v\right\}$, respectively. 
Let $\mathcal{L}^p([a,b];\mathbb{R}^n)$ be the family of $\mathbb{R}^n$-valued processes $\{f(t)\}_{a\le t\le b}$ which satisfies  $\mathcal{F}_t$-adapted and $\int_a^{b}\vert f(t)\vert^p {\rm{d}}t<\infty$, a.s. $\mathcal{L}^p(\mathbb{R}_+;\mathbb{R}^n)$ denotes the family of processes $\{f(t)\}_{t\ge 0}$ such that $\{f(t)\}_{0\le t\le T}\in \mathcal{L}^p([0,T];\mathbb{R}^n)$ for any $T>0$.
In this paper, we use $C$ represents a common positive number independent of $\Delta$, its value may vary with each appearance. 
 
Suppose $\alpha(t),t\ge 0,$ is a right-continuous Markov chain taking values in $S=\left\{1,2,\dots ,N\right\}$ with generator $\Gamma =(\gamma_{ij})_{N\times N}$, $\gamma_{ij}\ge 0$ denotes the transition rate from $i$ to $j$ when $i\neq j$, and
\begin{equation*}
\gamma_{ii}=-\sum_{j\neq i}\gamma_{ij}.
\end{equation*}
Assuming that $\alpha(\cdot)$ is independent of the Brownian motion $B(\cdot )$. 

Let $T>0$, consider the following SDEwMS
\begin{equation}\label{HSDE}
\begin{cases}
       {\rm{d}}z(t)=f(z(t),\alpha(t)){\rm{d}}t+g(z(t),\alpha(t)){\rm{d}}B(t), \\
       z(0)=z_0\in\mathbb{R}^n, \alpha(0)=i_0\in S,
\end{cases}
\end{equation}
 on $t\in [0,T]$, where
\begin{equation*}
f:\mathbb{R}^n\times S\to \mathbb{R}^n\quad {\rm{and}} \quad g:\mathbb{R}^n\times S\to \mathbb{R}^{n\times d}.
\end{equation*}
 We impose the following conditions:

\begin{assumption}\label{The global Lipschitz condition}
There is a number $L>0$ such that 
\begin{equation*}
\vert f(z,i)-f(\bar{z},i)\vert \vee\vert g(z,i)-g(\bar{z},i)\vert \le L\vert z-\bar{z}\vert,
\end{equation*} 
for all $i\in S$ and $z,\bar{z}\in\mathbb{R}^n$.
\end{assumption}

\begin{assumption}\label{initial condition}
There is a number $K>0$ such that 
\begin{equation*}
\vert f(0,i)\vert \vee\vert g(0,i)\vert \le K,~~ \forall i\in S.
\end{equation*} 
\end{assumption}

\begin{remark}\label{linear growth}
By Assumptions \ref{The global Lipschitz condition} and \ref{initial condition}, we can easily arrive at
\begin{equation*}
\vert f(z,i)\vert \vee\vert g(z,i)\vert \le (K+L)(1+\vert z\vert),
\end{equation*} 
for all $i\in S$ and $z\in\mathbb{R}^n$.
\end{remark}

\begin{lemma}[Lemma 2.2 in \cite{YUAN2004223}]\label{theorem_existence and uniqueness}
If Assumptions \ref{The global Lipschitz condition} and \ref{initial condition} hold, then for arbitrary $p\ge 2$, there is a unique solution for Eq.\eqref{HSDE} with the initial data $z_0$. In addition, the solution satisfies \begin{equation*}
\begin{split}
\mathbb{E}\left(\sup_{0\le t\le T}\vert z(t)\vert^p\right)\le C.
\end{split}
\end{equation*}
\end{lemma}


In the following, we will first introduce the classical methods for simulating discrete-time Markov chain that have been used in many papers, and further present the well known EM method.  In the next section, we will introduce the method used to simulate Markov chain in this paper, and further construct another type of EM method for SDEwMS which is different from the one given in the Ref.\cite{YUAN2004223}.

\subsection{The classical EM method}
The most commonly used method to generate the discrete Markov chain $\left\{\alpha_k^{\Delta},k=0,1,2,\dots\right\}$ is based on the properties of embedded discrete Markov chain:
For any given step size $\Delta>0$, let $\alpha_k^{\Delta}=\alpha(k\Delta)$ for $k\ge 0$. Then $\left\{\alpha_k^{\Delta} \right\}$ is a discrete Markov chain  with the one-step transition probability matrix 
\begin{equation*}
\mathbb{P}(\Delta)=(\mathbb{P}_{ij}(\Delta))_{N\times N}=e^{\Gamma \Delta}.
\end{equation*}
Hence, the discrete Markov chain $\left\{\alpha_k^{\Delta},k=0,1,2,\dots \right\}$ can be generated as follows: Let $\alpha_0^{\Delta}=i_0$ and compute a pseudo-random number $\zeta_1$ from the uniform $[0,1]$ distribution. Define 
\begin{equation*}
\alpha_1^{\Delta}=\begin{cases}i_1,\,\,&{\rm{if\,\,}} i_1\in S-\{N\} {\rm{\,\,such\,\, that\,\,}} \sum_{j=1}^{i_1-1}\mathbb{P}_{i_0,j}(\Delta)\le \zeta_1<\sum_{j=1}^{i_1}\mathbb{P}_{i_0,j}(\Delta),\\
N,\,\,&{\rm{if\,\,}} \sum_{j=1}^{N-1}\mathbb{P}_{i_0,j}(\Delta)\le \zeta_1,
\end{cases}
\end{equation*}
where we set $\sum_{j=1}^{0}\mathbb{P}_{i_0,j}(\Delta)=0$ as usual. Generally, having calculated $\alpha_0^{\Delta},\alpha_1^{\Delta},\dots,\alpha_k^{\Delta},$ we compute $\alpha_{k+1}^{\Delta}$ by drawing a uniform $[0,1]$ pseudo-random number $\zeta_{k+1}$ and setting 
\begin{equation*}
\alpha_{k+1}^{\Delta}=\begin{cases}i_{k+1},\,\,&{\rm{if\,\,}} i_{k+1}\in S-\{N\} {\rm{\,\,such\,\, that\,\,}} \\
&\quad\sum_{j=1}^{i_{k+1}-1}\mathbb{P}_{r_k^{\Delta},j}(\Delta)\le \zeta_{k+1}<\sum_{j=1}^{i_{k+1}}\mathbb{P}_{r_k^{\Delta},j}(\Delta),\\
N,\,\,&{\rm{if\,\,}} \sum_{j=1}^{N-1}\mathbb{P}_{\alpha_k^{\Delta},j}(\Delta)\le \zeta_{k+1}.
\end{cases}
\end{equation*}
After explaining how to get the Markov chain $\left\{\alpha_k^{\Delta},k=0,1,2,\dots \right\}$, we can now give the classical EM method for the SDEwMS \eqref{HSDE}. Given a step size $\Delta>0$, let $t_k=k\Delta$ for $k\ge 0$, setting $X_0=z_0,\alpha_0^{\D}=i_0$ and forming
\begin{equation}\label{classical EM}
X_{k+1}=X_k+f(X_k,\alpha_k^{\D})\D+g(X_k,\alpha_k^{\D})\D B_k,
\end{equation}
where $\D B_k=B(t_{k+1})-B(t_k)$, $X_k$ is the approximation of $z(t_k)$. Let 
\begin{equation*}
\bar{X}(t)=X_k,\quad \bar{\alpha}(t)=\alpha_k^{\D} \quad {\rm for}\quad t\in [t_k,t_{k+1}),
\end{equation*}
 and the continuous EM method is defined by
 \begin{equation}\label{classical continuous EM}
 X(t)=X_0+\int_0^t f(\bar{X}(s),\bar{\alpha}(s)){\rm d}s+\int_0^t g(\bar{X}(s),\bar{\alpha}(s)){\rm d}B(s).
  \end{equation}
  It can be verified that $X(t_k)=\bar{X}(t_k)=X_k$.
    
     \begin{remark} 
As we said in the Section \ref{sec1}, there are only a few papers that estimates the error between the numerical approximation and the exact solution for hybrid systems in the sense of $p$-th moment. Inequality \eqref{e: remark 1} is equivalent to 
     \begin{equation*}
 \left( \mathbb{E}\sup_{t\in[0,T]}|y(t)-x(t)|^p\right)^{1/p}\le C\D^{1/p},
  \end{equation*}
  where $C=(C_5(1+\mathbb{E}|x_0|^p))^{1/p}$, this implies the $L^p$-convergence order for the $\theta$ method to the hybrid system is $1/p$, instead of $1/2$, which is the convergence order of $\theta$ method for stochastic systems without Markov chain (\cite{2015Zong}). Main result in \cite{2020ZhangWei} is similar to the Theorem 3 in \cite{2019HouZhenting}.  \end{remark}
  
  In the next section, we will give a different EM scheme using another method to formulate the Markov chain $\alpha(t)$, and we will prove that the EM method given in this paper will converge to Eq.\eqref{HSDE} in the sense of $L^p~(p\ge 2)$ with the order $1/2$.
  
\section{Euler-Maruyama method}\label{Euler method}
For the generation of the Markov chain $\alpha(t)$, we cite the methodology of  formulating the Markov chain from Section 2.4 in \cite{GeorgeYinbook}. In order to get the sample paths of $\alpha(t)$, we need to determine the time of residence in each state and the succeeding actions. The chain remains at any given state $i_0 (i_0\in S)$ for a random length of time, $\tau_1$, which follows an exponential distribution with parameter  $-\gamma_{i_0 i_0}$, hence $\tau_1$ can be obtained by
\begin{equation*}
\tau_1=\frac{\log(1-\zeta_1)}{\gamma_{i_0 i_0}},
\end{equation*}
where $\zeta_1$ is a random variable uniformly distributed in $(0,1)$. Then, the process will enter another state. In addition, the probability that state $j$ (with $j\in S, j\neq i_0$) becomes the next residence of the chain is $\gamma_{i_0 j}/(-\gamma_{i_0 i_0})$. The position after the jump is determined by a discrete random variable $i_1 (i_1\in S\setminus\{i_0\})$, namely $\alpha(\tau_1)=i_1$. The value of $i_1$ is given by 
\begin{equation*}
i_1=\begin{cases}
      1,&\text{if}~\xi_1< \gamma_{i_01}/(-\gamma_{i_0i_0}),\\
      2, &\text{if}~\gamma_{i_01}/(-\gamma_{i_0i_0})\le \xi_1< (\gamma_{i_01}+\gamma_{i_02})/(-\gamma_{i_0i_0}),\\
      \vdots &\vdots\\
      N, &\text{if}~\sum_{j\neq i_0, j\le N-1}\gamma_{i_0j}/(-\gamma_{i_0i_0})< \xi_1,
\end{cases}
\end{equation*}
where $\xi_1$ is a random variable uniformly distributed in $(0,1)$.

The chain remains at state $i_1$ for a random length of time, $\tau_2$ , which follows an exponential distribution with parameter  $-\gamma_{i_1 i_1}$, thus
\begin{equation*}
\tau_2=\frac{\log(1-\zeta_2)}{\gamma_{i_1 i_1}},
\end{equation*}
where $\zeta_2$ is also a random variable uniformly distributed in $(0,1)$. Then, the process will enter another state. The post-jump location is identified by a discrete random variable $i_2 (i_2\in S\setminus\{i_1\})$, which implies $\alpha(\tau_1+\tau_2)=i_2$. The value of $i_2$ is determined by 
\begin{equation*}
i_2=\begin{cases}
      1,&\text{if}~\xi_2< \gamma_{i_11}/(-\gamma_{i_1i_1}),\\
      2, &\text{if}~\gamma_{i_11}/(-\gamma_{i_1i_1})\le \xi_2< (\gamma_{i_11}+\gamma_{i_12})/(-\gamma_{i_1i_1}),\\
      \vdots &\vdots\\
      N, &\text{if}~\sum_{j\neq i_1, j\le N-1}\gamma_{i_1j}/(-\gamma_{i_1i_1})< \xi_2,
\end{cases}
\end{equation*}
where $\xi_2$ is a random variable uniformly distributed in $(0,1)$. Therefore, repeating the procedure above, the sampling path of $\alpha(t),t\ge 0$ is composed of exponential random variables and $U(0,1)$ random variables alternately.

Recall that nearly all sample paths of $\alpha(\cdot)$ are right-continuous piecewise constant function with finite sample jumps in $[0,T]$. Thus, there are stopping times $0=\bar{\tau}_0<\bar{\tau}_1<\bar{\tau}_2<\cdots<\bar{\tau}_{\bar{N}}=T~(\bar{N}\in \mathbb{N}_+)$, where $\bar{\tau}_k=\sum_{j=1}^{k}\tau_j, k=1,2,\dots,\bar{N}-1$, such that 
\begin{equation*}
\alpha(t)=\sum_{k=0}^{\bar{N}-1}i_kI_{[\bar{\tau}_k,\bar{\tau}_{k+1})}(t).
\end{equation*}

Now we are in a position to define the EM method to SDEwMS \eqref{HSDE}. Given a step size $\Delta>0$, let $t_k=k\Delta$ $(k\in\mathbb{N})$ be the gridpoints.

Define 
\begin{equation*}
J_i=\begin{cases}
      0,&\text{if}~i=0, \\
       \inf\{t\in(J_{i-1},T]\mid \alpha(t)\neq \alpha(t^-)\}\wedge \inf\{t\in(J_{i-1},T]\mid t=t_k, k\in \mathbb{N}\}, &\text{if}~i\ge 1.
\end{cases}
\end{equation*}

According to the definition of $J_i$, it is easy to know that 
\begin{equation*}
\tilde{N}:=\#\{J_i, i=0,1,\dots\}\le [T/\Delta]+\bar{N}+1,
\end{equation*}
where $\#\{J_i\}$ denotes the number of elements in set $\{J_i\}$. Then we define the EM method to \eqref{HSDE} of the following type by setting $Z_0=z(0)=z_0$,
\begin{equation}\label{EM}
\begin{split}
Z_{k+1}=Z_{k}+\sum_{i=0}^{\tilde{N}-1}f\left(Z_{k},\alpha(J_i)\right)I_{[t_{k},t_{k+1})}(J_i)\D J_i+\sum_{i=0}^{\tilde{N}-1}g\left(Z_{k},\alpha(J_i)\right)I_{[t_{k},t_{k+1})}(J_i)\Delta B_{J_i},
\end{split}
\end{equation}
for $k\in\mathbb{N}$, where $\D J_i=J_{i+1}-J_i, \Delta B_{J_i}=B(J_{i+1})-B(J_i)$. $Z_{k}$ is the approximate value of $z(t_{k})$. 
Let
\begin{equation*}
\bar{Z}(t)=\sum_{k=0}^{\infty}Z_{k}I_{[t_{k},t_{k+1})}(t),
\end{equation*}
the continuous EM method is given by
\begin{equation}\label{continuous EM}
Z(t)=Z_0+\int_{0}^t f(\bar{Z}(s),\alpha(s)){\rm d}s+\int_{0}^t g(\bar{Z}(s),\alpha(s)){\rm d}B(s).
\end{equation}
It can be verified that $Z(t_{k})=\bar{Z}(t_{k})=Z_{k},~\forall k\ge 0$.

\rm
\section{Rate of the $L^p$-convergence for the EM method}\label{Rate of strong convergence}
Similar to Lemma 4.1 in \cite{Mao2006book}, we can easily obtain the following conclusion.

\begin{lemma}\label{bounded lemma}
Let Assumptions \ref{The global Lipschitz condition} and \ref{initial condition} hold. Then for any $\D\in(0,1]$ and $p\ge 2$, the EM method \eqref{continuous EM} has the following property
\begin{equation*}
\mathbb{E}\left(\sup_{0\le t\le T}\vert Z(t)\vert^p\right)\le C.
\end{equation*}
\end{lemma}
The proof is omitted because it is similar to that for Lemma 4.1 in \cite{Mao2006book}.

\begin{lemma}\label{X_X_D}
Suppose that Assumptions \ref{The global Lipschitz condition} and \ref{initial condition} hold. Then for any $p\ge 2$, 
\begin{equation*}
\sup_{0\le t\le T}\mathbb{E}\vert Z(t)-\bar{Z}(t)\vert^p\le C\D^{p/2}.
\end{equation*}
\end{lemma}
\begin{proof}
 For any $t\in [0,T]$, according to \eqref{continuous EM} and the basic inequality $(\vert u\vert+\vert v\vert)^p\le 2^{p-1}(\vert u\vert^p+\vert v\vert^p)$, $p\ge 2$, one has 
\begin{equation*}
\begin{split}
\mathbb{E}\vert Z(t)-\bar{Z}(t)\vert^p\le& 2^{p-1}\mathbb{E}\left\vert \int_{[t/{\D}]\D}^t f(\bar{Z}(s),\alpha(s)){\rm d}s\right\vert^p
+2^{p-1}\mathbb{E}\left\vert \int_{[t/{\D}]\D}^t g(\bar{Z}(s),\alpha(s)){\rm d}B(s)\right\vert^p.
\end{split}
\end{equation*}
Applying H\"older's inequality and Theorem 1.7.1 in \cite{Maobook}, one can arrive at
\begin{equation*}
\begin{split}
\mathbb{E}\vert Z(t)-\bar{Z}(t)\vert^p\le& C\D^{p-1}\mathbb{E} \int_{[t/{\D}]\D}^t \left\vert f(\bar{Z}(s),\alpha(s))\right\vert^p{\rm d}s+C\D^{p/2-1}\mathbb{E} \int_{[t/{\D}]\D}^t \left\vert g(\bar{Z}(s),\alpha(s))\right\vert^p {\rm d}s.
\end{split}
\end{equation*}
On the basis of Remark \ref{linear growth} and Lemma \ref{bounded lemma}, one has
\begin{equation*}
\begin{split}
\mathbb{E}\vert Z(t)-\bar{Z}(t)\vert^p\le& C\D^{p/2-1}(\D^{p/2}+1)\mathbb{E} \int_{[t/{\D}]\D}^t \left(1+\vert \bar{Z}(s)\vert^p\right){\rm d}s\\
\le& C\D^{p/2-1}(\D^{p/2}+1) \int_{[t/{\D}]\D}^t \left(1+\mathbb{E}\sup_{0\le s\le t}\vert Z(s)\vert^p\right){\rm d}s\\
\le &C\D^{p/2},
\end{split}
\end{equation*}
since $t\in[0,T]$ is arbitrary, the proof is completed.
\end{proof}


\begin{theorem}\label{Strong convergence}
Under Assumption \ref{The global Lipschitz condition}, for any $p\ge 2$, the EM method \eqref{continuous EM} has the property that
\begin{equation*}
\mathbb{E}\left(\sup_{0\le t\le T}\vert z(t)-Z(t)\vert^{p}\right)\le C\Delta^{p/2}.
\end{equation*}
\end{theorem}
\begin{proof}
Recall \eqref{HSDE} and \eqref{continuous EM}, for any $t\in [0,T]$, according to It\^o's formula, we have
\begin{equation*}
\begin{split}
|z(t)-Z(t)|^2=&\int_0^t 2(z(s)-Z(s))^{\rm T}\left(f(z(s),\alpha(s))-f(\bar{Z}(s),\alpha(s))\right){\rm d}s\\
&+\int_0^t \left|g(z(s),\alpha(s))-g(\bar{Z}(s),\alpha(s))\right|^2{\rm d}s\\
&+\int_0^t 2(z(s)-Z(s))^{\rm T}\left(g(z(s),\alpha(s))-g(\bar{Z}(s),\alpha(s))\right){\rm d}B(s).
\end{split}
\end{equation*}
Then for any $T_1\in [0,T]$, it is easily to get that
\begin{equation}\label{A*}
\begin{split}
&\mathbb{E}\left(\sup_{0\le t\le T_1}\vert z(t)-Z(t)\vert^{p}\right)\\
\le&3^{p/2}\underbrace{\mathbb{E}\left\{\sup_{0\le t\le T_1}\left(\int_0^t 2(z(s)-Z(s))^{\rm T}\left(f(z(s),\alpha(s))-f(\bar{Z}(s),\alpha(s))\right){\rm d}s\right)^{p/2}\right\}}_{A_1}\\
&+3^{p/2}\underbrace{\mathbb{E}\left(\sup_{0\le t\le T_1}\left(\int_0^t \left|g(z(s),\alpha(s))-g(\bar{Z}(s),\alpha(s))\right|^2{\rm d}s\right)^{p/2}\right\}}_{A_2}\\
&+3^{p/2}\underbrace{\mathbb{E}\left(\sup_{0\le t\le T_1}\left(\int_0^t 2(z(s)-Z(s))^{\rm T}\left(g(z(s),\alpha(s))-g(\bar{Z}(s),\alpha(s))\right){\rm d}B(s)\right)^{p/2}\right\}}_{A_3}.
\end{split}
\end{equation}
Using H\"older's inequality and Assumption \ref{The global Lipschitz condition}, one can deduce that
\begin{equation}\label{A_1}
\begin{split}
A_1\le& 2^{p/2}T_1^{p/2-1}\mathbb{E}\int_0^{T_1} |z(s)-Z(s)|^{p/2}\left|f(z(s),\alpha(s))-f(\bar{Z}(s),\alpha(s))\right|^{p/2}{\rm d}s\\
\le& 2^{p/2}T_1^{p/2-1}L^{p/2}\mathbb{E}\int_0^{T_1} |z(s)-Z(s)|^{p/2}\left|z(s)-\bar{Z}(s)\right|^{p/2}{\rm d}s\\
\le& C\mathbb{E}\int_0^{T_1} |z(s)-Z(s)|^{p}{\rm d}s+ C\mathbb{E}\int_0^{T_1}\left|Z(s)-\bar{Z}(s)\right|^{p}{\rm d}s,\\
\end{split}
\end{equation}
and
\begin{equation}\label{A_2}
\begin{split}
A_2\le& T_1^{p/2-1}\mathbb{E}\int_0^{T_1}\left|g(z(s),\alpha(s))-g(\bar{Z}(s),\alpha(s))\right|^{p}{\rm d}s\\
\le& T_1^{p/2-1}L^{p}\mathbb{E}\int_0^{T_1}\left|z(s)-\bar{Z}(s)\right|^{p}{\rm d}s\\
\le& C\mathbb{E}\int_0^{T_1} |z(s)-Z(s)|^{p}{\rm d}s+ C\mathbb{E}\int_0^{T_1}\left|Z(s)-\bar{Z}(s)\right|^{p}{\rm d}s.\\
\end{split}
\end{equation}
Applying Theorem 1.7.2 in \cite{Maobook}, together with Assumption \ref{The global Lipschitz condition}, we can obtain that
\begin{equation}\label{A_3}
\begin{split}
A_3\le& C\mathbb{E}\int_0^{T_1} |z(s)-Z(s)|^{p/2}\left|g(z(s),\alpha(s))-g(\bar{Z}(s),\alpha(s))\right|^{p/2}{\rm d}s\\
\le& C\mathbb{E}\int_0^{T_1} |z(s)-Z(s)|^{p}{\rm d}s+ C\mathbb{E}\int_0^{T_1}\left|Z(s)-\bar{Z}(s)\right|^{p}{\rm d}s.\\
\end{split}
\end{equation}
Substituting \eqref{A_1}-\eqref{A_3} into \eqref{A*}, yields
\begin{equation*}
\begin{split}
\mathbb{E}\left(\sup_{0\le t\le T_1}\vert z(t)-Z(t)\vert^{p}\right)\le& C\mathbb{E}\int_0^{T_1} |z(t)-Z(t)|^{p}{\rm d}t+ C\mathbb{E}\int_0^{T_1}\left|Z(t)-\bar{Z}(t)\right|^{p}{\rm d}t.\\
\end{split}
\end{equation*}
By Lemma \ref{X_X_D} one can further show that
\begin{equation*}
\begin{split}
\mathbb{E}\left(\sup_{0\le t\le T_1}\vert z(t)-Z(t)\vert^{p}\right)\le& C\int_0^{T_1} \mathbb{E}\left(\sup_{0\le s\le t}|z(s)-Z(s)|^{p}\right){\rm d}t+ C\D^{p/2}.\\
\end{split}
\end{equation*}
It therefore follows from the Gronwall inequality that
\begin{equation*}
\mathbb{E}\left(\sup_{0\le t\le T_1}\vert z(t)-Z(t)\vert^{p}\right)\le  C{\rm e}^{CT_1}\D^{p/2},
\end{equation*}
since $T_1\in[0,T]$ is arbitrary, hence
\begin{equation*}
\mathbb{E}\left(\sup_{0\le t\le T}\vert z(t)-Z(t)\vert^{p}\right)\le  C\D^{p/2}.
\end{equation*}
The proof is completed.
\end{proof}

\section{Conclusions}\label{Conclusions}

In this paper, we develop the EM scheme, which is different from the one given in the references (such as \cite{YUAN2004223,MAO2007936}), to generate the approximate solutions to a class of SDEwMSs, and analyze the order of the errors in the $L^p$-sense. It has been proved that the $L^p$-convergence rate of the EM scheme given in this paper for SDEwMSs can reach $1/2$. We further point out that the techniques used in this paper to construct the EM method can also be used to construct other schemes for hybrid systems, such as stochastic theta method, tamed EM method, and Milstein method, etc. We believe that this approach will also contribute to the $L ^p$-convergence order of these numerical methods for hybrid systems.

\bibliography{mybibfile}

\end{document}